\documentclass[a4paper,12pt,twoside,leqno,final]{amsart}
\usepackage{amsmath}
\usepackage{amssymb}

\newtheorem{thm}{Theorem}[section]
\newtheorem{lem}[thm]{Lemma}

\newtheorem{prop}[thm]{Proposition}
\newtheorem{cor}[thm]{Corollary}

\newtheorem*{tha}{Theorem A}
\newtheorem*{thb}{Theorem B}

\theoremstyle{definition}
\newtheorem{exmp}{Example}[section]

\newcommand{\C}{{\mathbb C}}
\newcommand{\D}{{\mathbb D}}
\newcommand{\R}{{\mathbb R}}
\newcommand{\T}{{\mathbb T}}
\newcommand{\Z}{{\mathbb Z}}

\newcommand{\La}{\Lambda}

\newcommand{\f}{\frac}
\newcommand{\ov}{\overline}
\newcommand{\al}{\alpha}
\newcommand{\be}{\beta}

\newcommand{\ga}{\gamma}

\newcommand{\de}{\delta}
\newcommand{\la}{\lambda}
\newcommand{\ze}{\zeta}

\newcommand{\ph}{\varphi}

\newcommand{\const}{\text{\rm const}}

\numberwithin{equation}{section}

\title[Lacunary polynomials]
{Lacunary polynomials in $L^1$: geometry\\ 
of the unit sphere}
\author{Konstantin M. Dyakonov}
\address{Departament de Matem\`atiques i Inform\`atica, IMUB, BGSMath, Universitat de Barcelona, Gran Via 585, E-08007 Barcelona, Spain}
\address{ICREA, Pg. Llu\'is Companys 23, E-08010 Barcelona, Spain}
\email{konstantin.dyakonov@icrea.cat}
\keywords{Lacunary polynomial, Hardy space, extreme point, exposed point}
\subjclass[2010]{30C10, 30H10, 42A05, 46A55}
\thanks{Supported in part by grant MTM2017-83499-P from El Ministerio de Econom\'ia y Competitividad (Spain) and grant 2017-SGR-358 from AGAUR (Generalitat de Catalunya).}

\begin{document}
\begin{abstract}
Let $\Lambda$ be a finite set of nonnegative integers, and let $\mathcal P(\Lambda)$ be the linear hull of the monomials $z^k$ with $k\in\Lambda$, viewed as a subspace of $L^1$ on the unit circle. We characterize the extreme and exposed points of the unit ball in $\mathcal P(\Lambda)$.
\end{abstract}

\maketitle

\section{Introduction} 

Let $\mathcal P_N$ stand for the set of polynomials (in one complex variable) of degree at most $N$. By a {\it lacunary polynomial}, or {\it fewnomial}, in $\mathcal P_N$ one loosely means a polynomial therein that has \lq\lq gaps," in the sense that it is spanned by some selected monomials from the family $\{z^k:\,k=0,\dots,N\}$ rather than by all of them. Our plan is to look at the space of fewnomials generated by an arbitrary collection $\{z^k:\,k\in\La\}$ with $\La\subset\{0,\dots,N\}$, endow it with a suitable norm, and study the geometry of its unit sphere. 

\par To be precise, suppose that $N$ and $k_1,\dots,k_M$ are positive integers satisfying 
$$k_1<k_2<\dots<k_M<N.$$
We then consider the set
\begin{equation}\label{eqn:deflam}
\La:=\{0,\dots,N\}\setminus\{k_1,\dots,k_M\}
\end{equation}
and define $\mathcal P(\La)$ as the space of polynomials of the form $\sum_{k\in\La}c_kz^k$, with complex coefficients $c_k$. In the special case where $M=0$, the \lq\lq forbidden set" $\{k_1,\dots,k_M\}$ is empty and $\mathcal P(\La)$ reduces to $\mathcal P_N$. 

\par We shall restrict our polynomials to the circle 
$$\T:=\{\ze\in\C:|\ze|=1\}$$
(without forgetting that they actually live on $\C$) and embed $\mathcal P(\La)$ in $L^1=L^1(\T)$, the space of Lebesgue integrable complex-valued functions on $\T$, with norm 
\begin{equation}\label{eqn:lonenorm}
\|f\|_1:=\f 1{2\pi}\int_\T|f(\ze)|\,|d\ze|.
\end{equation}
With a function $f\in L^1$ we associate the sequence of its {\it Fourier coefficients}
$$\widehat f(k):=\f 1{2\pi}\int_\T\ov\ze^kf(\ze)\,|d\ze|,\qquad k\in\Z,$$
and the set 
$$\text{\rm spec}\,f:=\{k\in\Z:\,\widehat f(k)\ne0\},$$
called the {\it spectrum} of $f$. Thus, 
$$\mathcal P(\La)=\{f\in L^1:\,\text{\rm spec}\,f\subset\La\}.$$

\par Next we recall, in the framework of a general (complex) Banach space $X=(X,\|\cdot\|)$, the geometric concepts we shall be concerned with. We write  
$$\text{\rm ball}(X):=\{x\in X:\,\|x\|\le1\}$$
for the closed unit ball of $X$. As usual, a point in $\text{\rm ball}(X)$ is said to be {\it extreme} for the ball if it is not an interior point of any line segment contained in $\text{\rm ball}(X)$. Also, an element $\xi$ of $\text{\rm ball}(X)$ is called an {\it exposed point} thereof if there exists a functional $\phi\in X^*$ of norm $1$ for which 
$$\{x\in\text{\rm ball}(X):\,\phi(x)=1\}=\{\xi\}$$
(so that $\xi$ is the unique point of contact between the ball and the appropriate hyperplane). Clearly, every exposed point is extreme, and every extreme point lies on the sphere $\{x\in X:\,\|x\|=1\}$. 

\par Our goal here is to determine the two types of points in the unit sphere of $\mathcal P(\La)=(\mathcal P(\La),\|\cdot\|_1)$ for an arbitrary set $\La$ as above. This will be accomplished in subsequent sections. 

\par Meanwhile, we proceed with a brief overview of what happens in other important---and better studied---subspaces of $L^1$. First of all, $\text{\rm ball}(L^1)$ has no extreme points and hence no exposed points. Next, we take a look at the {\it Hardy space} 
$$H^1:=\{f\in L^1:\,\text{\rm spec}\,f\subset[0,\infty)\}.$$
Equivalently, $H^1$ is formed by the $L^1$ functions whose Poisson integral (i.e., harmonic extension) is holomorphic on the disk 
$$\D:=\{z\in\C:\,|z|<1\}.$$
Recall that an $H^1$ function is said to be {\it inner} if it has modulus $1$ a.e. on $\T$, while the non-null functions $f\in H^1$ satisfying 
$$\log|\widehat f(0)|=\f 1{2\pi}\int_\T\log|f(\ze)|\,|d\ze|$$
are termed {\it outer}. We refer to \cite[Chapter II]{G} for these concepts and the theory around them, including the inner-outer factorization, etc.

\par Now, it was proved by de Leeuw and Rudin in \cite{dLR} (see also \cite[Chapter IV]{G}) that the extreme points of $\text{\rm ball}(H^1)$ are precisely the outer functions $f\in H^1$ with $\|f\|_1=1$. By contrast, it is far from clear which unit-norm (and outer) functions in $H^1$ arise as exposed points therein. Such functions are also known as {\it rigid}; they turn up in various connections (e.g., in relation to Toeplitz operators, see \cite{Hay, Sar}) and have attracted quite a bit of attention. A number of their properties are established in \cite{Pol, Sar1, Sar2} and \cite[Section 3]{DJFA}. 

\par We further discuss an extension of the de Leeuw--Rudin result, which is relevant to our topic. Let $\ph$ be an essentially bounded function on $\T$, and put 
\begin{equation}\label{eqn:kertoepphi}
K_1(\ph):=\{f\in H^1:\,\ov{z\ph f}\in H^1\}.
\end{equation}
We note that $K_1(\ph)$ is actually the kernel in $H^1$ of the Toeplitz operator with symbol $\ph$, a map we do not formally define here. Assume, in addition, that $K_1(\ph)\ne\{0\}$. (In particular, this happens if $\ov\ph$ is a nonconstant inner function, in which case $K_1(\ph)$ becomes the corresponding {\it model subspace}. Also, if $\ph\equiv0$ then $K_1(\ph)$ is the whole of $H^1$.) Now, by \cite[Theorem 6]{DPAMS}, a function $f\in K_1(\ph)$ with $\|f\|_1=1$ is an extreme point of $\text{\rm ball}(K_1(\ph))$ if and only if 
\begin{equation}\label{eqn:innrelprime}
\text{\rm the inner factors of $f$ and $\ov{z\ph f}$ are relatively prime}
\end{equation}
(meaning that there is no nonconstant inner function $J$ such that $f/J\in H^1$ and $\ov{z\ph f}/J\in H^1$). For model subspaces in $H^1$, this was established earlier in \cite{DKal}. Yet another de Leeuw--Rudin type theorem related to Toeplitz kernels can be found in \cite{DAMP}. As far as exposed points of $\text{\rm ball}(K_1(\ph))$ are concerned, no nice description is available; see, however, \cite[Section 2]{DSib} for some partial results in the model subspace setting. 

\par Finally, we turn to $\mathcal P_N$, the space of polynomials $p$ with $\deg p\le N$, viewed again as a subspace of $L^1$. One easily checks that $\mathcal P_N$ is a special case of $K_1(\ph)$, as defined by \eqref{eqn:kertoepphi}, with $\ph(z)=\ov z^{N+1}$. Therefore, to decide whether a unit-norm polynomial $p\in\mathcal P_N$ is an extreme point of $\text{\rm ball}(\mathcal P_N)$, we only have to rewrite condition \eqref{eqn:innrelprime} with this specific $\ph$ plugged in, and with $p$ in place of $f$. The role of the function $\ov{z\ph f}$ goes then to the polynomial $p^*\in\mathcal P_N$ defined (on $\T$) by the formula $p^*=z^N\ov p$, so that 
\begin{equation}\label{eqn:pstar}
p^*(z):=z^N\ov{p\left(1/\ov z\right)},\qquad z\in\C\setminus\{0\}.
\end{equation}
Equivalently, if 
\begin{equation}\label{eqn:psumckzk}
p(z)=\sum_{k=0}^Nc_kz^k,
\end{equation}
then $p^*(z)=\sum_{k=0}^N\ov c_kz^{N-k}$. The next fact now comes out readily. 

\begin{tha} Suppose that $p\in\mathcal P_N$ and $\|p\|_1=1$. The following conditions are equivalent. 
\par{\rm (i.A)} $p$ is an extreme point of $\text{\rm ball}(\mathcal P_N)$. 
\par{\rm (ii.A)} The polynomials $p$ and $p^*$ have no common zeros in $\D$.
\end{tha}

It should be noted that whenever $p(a)=0$ for a point $a\in\C\setminus\{0\}$, it follows that $p^*(1/\ov a)=0$, the multiplicities of these respective zeros being the same. Also, we have $p(0)=0$ if and only if $\deg p^*<N$; moreover, the multiplicity of zero at $0$ will be at least $\ell$ if and only if $\deg p^*\le N-\ell$. The roles of $p$ and $p^*$ are of course interchangeable in these statements, since $(p^*)^*=p$. 

\par This said, we can rephrase condition (ii.A) in terms of $p$ alone. Namely, a polynomial \eqref{eqn:psumckzk} satisfies (ii.A) if and only if it has the properties that 
$$|c_0|+|c_N|\ne0$$
and 
\begin{equation}\label{eqn:nosymmzeros}
\text{\rm $p$ has no pair of symmetric zeros with respect to $\T$}
\end{equation}
(i.e., there is no point $a\in\D\setminus\{0\}$ for which $p(a)=p(1/\ov a)=0$). It was in this form that Theorem A appeared in \cite[Section 3]{DMRL2000}, where the next result was also established. 

\begin{thb} Suppose that $p\in\mathcal P_N$ and $\|p\|_1=1$. The following are equivalent. 
\par{\rm (i.B)} $p$ is an exposed point of $\text{\rm ball}(\mathcal P_N)$. 
\par{\rm (ii.B)} Condition {\rm (ii.A)} is fulfilled, and the zeros of $p$ lying on $\T$ (if any) are all simple.
\end{thb}

Thus, in particular, (ii.B) includes a stronger version of property \eqref{eqn:nosymmzeros}: this time, even \lq\lq degenerate pairs" of symmetric zeros (i.e., multiple zeros on $\T$) are forbidden for $p$. 

\par We mention in passing that Theorem A has a counterpart in the case where $\mathcal P_N$ is endowed with the supremum norm $\|\cdot\|_\infty$. Indeed, the extreme points of the unit ball in $(\mathcal P_N,\|\cdot\|_\infty)$ were characterized by the author in \cite{DMRL2003}. No such counterpart of Theorem B seems to be currently available, except for a partial result (see \cite[Theorem 1.2]{Neu}) that settles the case of $(\mathcal P(\La),\|\cdot\|_\infty)$ for a three-point set $\La$. 

\par Going back to the $L^1$ setting, we now seek to extend Theorems A and B to the \lq\lq lacunary" situation, where $\mathcal P_N$ gets replaced by $\mathcal P(\La)$. When $\La$ actually has gaps, so that the set $\{k_1,\dots,k_M\}$ in \eqref{eqn:deflam} is nonempty, the space $\mathcal P(\La)$ is no longer writable as $K_1(\ph)$ and a new method is needed. Anyhow, one quick observation is that whenever a polynomial $p\in\mathcal P(\La)$ is an extreme (resp., exposed) point of $\text{\rm ball}(\mathcal P_N)$, it will be also extreme (resp., exposed) for $\text{\rm ball}(\mathcal P(\La))$. Thus, the interesting case is that where condition (ii.A) (resp., (ii.B)) is violated. 

\par There are several types of difficulties we have to face when moving from $\mathcal P_N$ to $\mathcal P(\La)$. Most notably, dividing a polynomial from $\mathcal P(\La)$ by one of its elementary factors, we cannot expect the quotient to be in $\mathcal P(\La)$. For instance, if $\La_N:=\{0,N\}$ with an integer $N\ge2$, then the function $z\mapsto1-z^N$ is in $\mathcal P(\La_N)$, but the ratio 
$$\f{1-z^N}{1-z}=1+z+\dots+z^{N-1}$$
is not. For similar reasons, a polynomial from $\mathcal P(\La)$ that vanishes at a point $a\in\D$ need not be divisible by the elementary Blaschke factor $(z-a)/(1-\ov az)$; again, the quotient may well be in $\mathcal P_N\setminus\mathcal P(\La)$. Nor is $\mathcal P(\La)$ preserved by the map $p\mapsto p^*$, except when $\La$ has the appropriate symmetry property. 

\par Besides, the $L^1$ norms of lacunary polynomials tend to be rather unhandy. To illustrate this by a historical example, we mention the notoriously wicked conjecture of Littlewood concerning the magnitude of $\left\|\sum_{k\in\Lambda}z^k\right\|_1$, with $\La$ as above. Specifically, it was conjectured---and eventually proved, after a few decades---that this quantity is bounded below by an absolute constant times $\log(\#\La)$. Two different proofs were given, one in \cite{Kon} and the other in \cite{MPS}. (The latter actually deals with generic lacunary polynomials; see also \cite{DArk} for related results.) 

\par The plan for the rest of the paper is as follows. In Sections 2 and 3, we state our main theorems that characterize the extreme and exposed points, respectively, among the unit-norm polynomials in $\mathcal P(\La)$. In both cases, the statements are followed by a brief discussion, and some examples are provided. In Section 4, we collect a few auxiliary lemmas to lean upon later. Finally, we prove our results in Sections 5 and 6. 

\par We conclude this introduction with a couple of open questions that puzzle us. First, what happens to our results in higher dimensions (i.e., with $\T^d$ in place of $\T$, the role of $\La$ being played by a finite set of multi-indices in $\Z^d$)? Second, what about Paley--Wiener type analogues of $\mathcal P(\La)$ on the real line? Here, the idea would be to replace $\La$ by a compact set $K\subset\R$ and consider the space of all entire functions $f$ with $\int_\R|f(x)|\,dx<\infty$ whose Fourier transform is supported on $K$. So far, only the classical---i.e., nonlacunary---case (where $K$ is an interval) has been studied in this connection; see \cite[Section 5]{DMRL2000}. 

\section{Extreme points: criterion and examples}

Throughout, $\La$ will be a fixed set of the form \eqref{eqn:deflam} and $\mathcal P(\La)$ will be viewed as a subspace of $L^1$, normed by \eqref{eqn:lonenorm}. To determine whether a unit-norm polynomial $p$ from $\mathcal P(\La)$ is an extreme point of the unit ball, we first cook up a certain matrix $\mathfrak M=\mathfrak M(p)$ from it; the construction will be described in a moment. This done, the answer will be stated in terms of the rank of $\mathfrak M$. 

\par Given a polynomial $p\in\mathcal P(\La)$ with $\|p\|_1=1$, consider also the associated polynomial $p^*(\in\mathcal P_N)$ defined by \eqref{eqn:pstar}. Now let $a_1,\dots,a_n$ be an enumeration of the common zeros of $p$ and $p^*$ lying in $\D$. It is understood that the $a_j$'s are pairwise distinct, and we attach to them the positive integers
$$m_j:=\min\left\{\text{\rm mult}(a_j,p),\,\text{\rm mult}(a_j,p^*)\right\},\qquad j=1,\dots,n,$$
where $\text{\rm mult}(a_j,\cdot)$ is the multiplicity of zero at $a_j$ for the polynomial in question. Finally, we put 
$$m:=\sum_{j=1}^nm_j\qquad\text{\rm and}\qquad s:=N-2m.$$ 
\par We then introduce the polynomial
\begin{equation}\label{eqn:ggg}
G(z):=\prod_{j=1}^n(z-a_j)^{m_j}(1-\ov a_jz)^{m_j}
\end{equation}
(which clearly divides both $p$ and $p^*$) along with the ratio
\begin{equation}\label{eqn:rrr}
R(z):=p(z)/G(z),
\end{equation}
which is also a polynomial. The factorization $p=GR$ that arises will occasionally be referred to as {\it canonical}. Furthermore, we observe that 
\begin{equation}\label{eqn:degrles}
\deg R\le s.
\end{equation}
Indeed, if none of the $a_j$'s is zero, then \eqref{eqn:degrles} follows from the facts that $\deg p\le N$ and $\deg G=2m$. Otherwise, we may assume that $a_1=0$, in which case $\deg p\le N-m_1$ (because $p^*$ is divisible by $z^{m_1}$) and $\deg G=2m-m_1$, so \eqref{eqn:degrles} is again valid. 
\par Letting 
\begin{equation}\label{eqn:defofck}
C_k:=\widehat R(k),\qquad k\in\Z, 
\end{equation}
we therefore have
\begin{equation}\label{eqn:rofz}
R(z)=\sum_{k=0}^sC_kz^k,
\end{equation}
while $C_k=0$ for all $k\in\Z\setminus[0,s]$. We now define 
\begin{equation}\label{eqn:defakbk}
A(k):=\text{\rm Re}\,C_k,\qquad B(k):=\text{\rm Im}\,C_k\qquad(k\in\Z)
\end{equation}
and consider, for $j=1,\dots,M$ and $l=0,\dots,m$, the numbers 
\begin{equation}\label{eqn:abplus}
A^+_{j,l}:=A(k_j+l-m)+A(k_j-l-m),\qquad B^+_{j,l}:=B(k_j+l-m)+B(k_j-l-m)
\end{equation}
and
\begin{equation}\label{eqn:abminus}
A^-_{j,l}:=A(k_j+l-m)-A(k_j-l-m),\qquad B^-_{j,l}:=B(k_j+l-m)-B(k_j-l-m),
\end{equation}
the integers $k_j$ being the same as in \eqref{eqn:deflam}. From these, we build the $M\times(m+1)$ matrices 
\begin{equation}\label{eqn:plusmat}
\mathcal A^+:=\left\{A^+_{j,l}\right\},\qquad\mathcal B^+:=\left\{B^+_{j,l}\right\}
\end{equation}
and the $M\times m$ matrices
\begin{equation}\label{eqn:minusmat}
\mathcal A^-:=\left\{A^-_{j,l}\right\},\qquad
\mathcal B^-:=\left\{B^-_{j,l}\right\}.
\end{equation}
Here, the row index $j$ always runs from $1$ to $M$, while the column index $l$ runs from $0$ to $m$ for each of the \lq\lq plus-matrices" \eqref{eqn:plusmat}, and from $1$ to $m$ for each of the \lq\lq minus-matrices" \eqref{eqn:minusmat}. 
\par Finally, we need the block matrix 
\begin{equation}\label{eqn:blockmatrix}
\mathfrak M:=
\begin{pmatrix}
\mathcal A^+ & \mathcal B^- \\
\mathcal B^+ & -\mathcal A^-
\end{pmatrix},
\end{equation}
which has $2M$ rows and $2m+1$ columns. 

\begin{thm}\label{thm:extremepoints} Suppose that $p\in\mathcal P(\La)$ and $\|p\|_1=1$. Then $p$ is an extreme point of $\text{\rm ball}(\mathcal P(\La))$ if and only if $\text{\rm rank}\,\mathfrak M=2m$.
\end{thm}

\par Now, since $\text{\rm rank}\,\mathfrak M\le\min(2M,2m+1)$, the following consequence is immediate. 

\begin{cor}\label{cor:immcor} If $p$ is a unit-norm polynomial in $\mathcal P(\La)$ satisfying $M<m$, then $p$ is a non-extreme point of $\text{\rm ball}(\mathcal P(\La))$. 
\end{cor}

\par Roughly speaking, this means that if $p$ is not too lacunary (in the sense that $M$, the number of forbidden frequencies in \eqref{eqn:deflam}, is not too large), then the situation is similar to that in the nonlacunary case, as described by Theorem A. 

\par On the other hand, whenever a unit-norm polynomial $p\in\mathcal P(\La)$ happens to be an extreme point of $\text{\rm ball}(\mathcal P_N)$, it will be extreme for $\text{\rm ball}(\mathcal P(\La))$ as well. Of course, this fact follows at once from the inclusion $\mathcal P(\La)\subset\mathcal P_N$, but we can also verify it by comparing the characterizations from Theorem A and Theorem \ref{thm:extremepoints}. Indeed, in this case $p$ and $p^*$ have no common zeros in $\D$, whence $m=0$. Accordingly, the polynomials \eqref{eqn:ggg} and \eqref{eqn:rrr} in the canonical factorization take the form $G=1$ and $R=p$. In particular, the coefficients $C_{k_j}=A(k_j)+iB(k_j)$ in \eqref{eqn:rofz} are then null for $j=1,\dots,M$. This means that the blocks $\mathcal A^+$ and $\mathcal B^+$ in \eqref{eqn:blockmatrix} reduce to zero columns, whereas the other two blocks are absent, so we have $\text{\rm rank}\,\mathfrak M=0(=2m)$. 

\par More interesting examples are to be found among those polynomials which are non-extreme points of $\text{\rm ball}(\mathcal P_N)$ and satisfy $M\ge m$. Two such instances will now be considered. 

\begin{exmp} Let 
$$p(z)=\ga\left(z-\f12\right)(2-z)\left(1+z^4\right),$$
where $\ga>0$ is the number that ensures $\|p\|_1=1$. Clearly, $p\in\mathcal P_6$ and $\widehat p(3)=0$, so that $p\in\mathcal P(\La)$ with $\La=\{0,1,2,4,5,6\}$. This last set can obviously be written as \eqref{eqn:deflam}, where $N=6$, $M=1$ and $k_1=3$. Also, since $p\left(\f12\right)=p(2)=0$ (both zeros being simple), the polynomials $p$ and $p^*$ have a common zero (of multiplicity $1$) at the point $a_1=\f12$; moreover, they have no other common zeros in $\D$. The corresponding parameters are, therefore, $n=m_1=m=1$. The canonical factorization $p=GR$, as determined by \eqref{eqn:ggg} and \eqref{eqn:rrr}, is in this case obtained by taking 
\begin{equation}\label{eqn:gggonehalf}
G(z)=\left(z-\f12\right)\left(1-\f12z\right)
\end{equation}
and
$$R(z)=2\ga\left(1+z^4\right).$$
Recalling the notations \eqref{eqn:defofck} and \eqref{eqn:defakbk}, we now have $C_k=0$ for all $k\in\Z\setminus\{0,4\}$, while $C_0=C_4=2\ga$. Equivalently, $A(k)=0$ for all $k\in\Z\setminus\{0,4\}$ and 
$$A(0)=A(4)=2\ga,$$ 
whereas $B(k)=0$ for all $k\in\Z$. It follows immediately that the matrix elements \eqref{eqn:abplus} with $(j,l)\in\{(1,0),(1,1)\}$, as well as \eqref{eqn:abminus} with $(j,l)=(1,1)$, are all null. Consequently, $\mathfrak M$ is the zero matrix (of size $2\times3$) and its rank is $0$. This number being different from $2m(=2)$, we see from Theorem \ref{thm:extremepoints} that $p$ is a non-extreme point of $\text{\rm ball}(\mathcal P(\La))$. 
\end{exmp}

\begin{exmp} Now let 
$$p(z)=\de\left(z-\f12\right)\left(8-z^3\right),$$
where $\de>0$ is the normalizing constant that ensures $\|p\|_1=1$. This time, we have $p\in\mathcal P_4$ and $\widehat p(2)=0$, so that $p\in\mathcal P(\La)$ with $\La=\{0,1,3,4\}$. We also write this set $\La$ in the form \eqref{eqn:deflam}, putting $N=4$, $M=1$ and $k_1=2$. As in the previous example, the polynomials $p$ and $p^*$ have a common zero (of multiplicity $1$) at the point $a_1=\f12$ and no other common zeros in $\D$. Thus, $n=m_1=m=1$. As regards the canonical factorization $p=GR$, one factor is again given by \eqref{eqn:gggonehalf}, while the other is 
$$R(z)=2\de\left(4+2z+z^2\right).$$
The coefficients \eqref{eqn:defofck} are thereby known, as are the numbers \eqref{eqn:defakbk}, and we use these to compute the matrix entries \eqref{eqn:abplus} and \eqref{eqn:abminus}. Eventually, we find that 
$$\mathfrak M:=2\de
\begin{pmatrix}
4 & 5 & 0 \\
0 & 0 & 3
\end{pmatrix},
$$
and so $\text{\rm rank}\,\mathfrak M=2(=2m)$. Consequently, Theorem \ref{thm:extremepoints} tells us that $p$ is an extreme point of $\text{\rm ball}(\mathcal P(\La))$. 
\end{exmp}

\section{Exposed points: criterion and examples}

Now we move on to describing the exposed points of $\text{\rm ball}(\mathcal P(\La))$. The criterion, to be stated in terms of the appropriate matrix $\widetilde{\mathfrak M}$ built from the polynomial in question, will be close in spirit to Theorem \ref{thm:extremepoints}. However, there are some adjustments to be made and some complications to be dealt with. 
\par Once again, we fix a unit-norm element $p$ of $\mathcal P(\La)$ and we recall the canonical factorization $p=GR$ from the preceding section, the two factors being given by \eqref{eqn:ggg} and \eqref{eqn:rrr}. Furthermore, let $\ze_1,\dots,\ze_\nu$ be the distinct zeros of $p$ lying on $\T$, and let $\la_1,\dots,\la_\nu$ be their respective multiplicities. We also need the (nonnegative) integers 
\begin{equation}\label{eqn:mujlaj}
\mu_j:=[\la_j/2],\qquad j=1,\dots,\nu,
\end{equation}
where $[\cdot]$ denotes integral part, as well as the numbers
\begin{equation}\label{eqn:mums}
\mu:=\sum_{j=1}^\nu\mu_j,\qquad\widetilde m:=m+\mu,\qquad\widetilde s:=N-2\widetilde m.
\end{equation}
We then define the polynomials $G_0$ and $\widetilde G$ by putting 
\begin{equation}\label{eqn:gnought}
G_0(z):=\prod_{j=1}^\nu(z-\ze_j)^{\mu_j}(1-\ov\ze_jz)^{\mu_j}
\end{equation}
and 
\begin{equation}\label{eqn:newfactone}
\widetilde G:=GG_0.
\end{equation}
Rewriting \eqref{eqn:gnought} in the form
\begin{equation}\label{eqn:gnoughtbis}
G_0(z)=\prod_{j=1}^\nu(-\ov\ze_j)^{\mu_j}(z-\ze_j)^{2\mu_j}
\end{equation}
and noting that $0\le2\mu_j\le\la_j$, we see that $p$ is divisible by $G_0$ and hence also by $\widetilde G$ (because $G$ and $G_0$ are relatively prime). The function
\begin{equation}\label{eqn:newfacttwo}
\widetilde R:=p/\widetilde G
\end{equation}
is therefore a polynomial; moreover, since $\widetilde R=R/G_0$ and $\deg G_0=2\mu$, we deduce from \eqref{eqn:degrles} that 
\begin{equation}\label{eqn:degrtillestil}
\deg\widetilde R\le s-2\mu=N-2m-2\mu=\widetilde s.
\end{equation}

\par Our further steps towards constructing the matrix $\widetilde{\mathfrak M}$ are quite similar to what we did previously to arrive at \eqref{eqn:blockmatrix}. Namely, 
we write $\widetilde C_k$ (with $k\in\Z$) for the $k$th Fourier coefficient of $\widetilde R$, so that 
\begin{equation}\label{eqn:rofztil}
\widetilde R(z)=\sum_{k=0}^{\widetilde s}\widetilde C_kz^k,
\end{equation}
and put
\begin{equation}\label{eqn:defakbktil}
\widetilde A(k):=\text{\rm Re}\,\widetilde C_k,
\qquad\widetilde B(k):=\text{\rm Im}\,\widetilde C_k\qquad(k\in\Z).
\end{equation}
Next we define, for $j=1,\dots,M$ and $l=0,\dots,\widetilde m$, the numbers 
\begin{equation}\label{eqn:abplustil}
\widetilde A^+_{j,l}:=\widetilde A(k_j+l-\widetilde m)+\widetilde A(k_j-l-\widetilde m),\qquad
\widetilde B^+_{j,l}:=\widetilde B(k_j+l-\widetilde m)+\widetilde B(k_j-l-\widetilde m)
\end{equation}
and
\begin{equation}\label{eqn:abminustil}
\widetilde A^-_{j,l}:=\widetilde A(k_j+l-\widetilde m)-\widetilde A(k_j-l-\widetilde m),\qquad
\widetilde B^-_{j,l}:=\widetilde B(k_j+l-\widetilde m)-\widetilde B(k_j-l-\widetilde m).
\end{equation}
This done, we build the $M\times(\widetilde m+1)$ matrices 
\begin{equation}\label{eqn:plusmattil}
\widetilde{\mathcal A}^+:=\left\{\widetilde A^+_{j,l}\right\},\qquad
\widetilde{\mathcal B}^+:=\left\{\widetilde B^+_{j,l}\right\}
\end{equation}
and the $M\times\widetilde m$ matrices
\begin{equation}\label{eqn:minusmattil}
\widetilde{\mathcal A}^-:=\left\{\widetilde A^-_{j,l}\right\},\qquad
\widetilde{\mathcal B}^-:=\left\{\widetilde B^-_{j,l}\right\}.
\end{equation}
Here, the row index $j$ always runs from $1$ to $M$, while the column index $l$ runs from $0$ to $\widetilde m$ for each of the two matrices in \eqref{eqn:plusmattil}, and from $1$ to $\widetilde m$ for each of those in \eqref{eqn:minusmattil}. Finally, the block matrix $\widetilde{\mathfrak M}$, with $2M$ rows and $2\widetilde m+1$ columns, is defined by 
\begin{equation}\label{eqn:blockmatrixtil}
\widetilde{\mathfrak M}:=
\begin{pmatrix}
\widetilde{\mathcal A}^+ & \widetilde{\mathcal B}^- \\
\widetilde{\mathcal B}^+ & -\widetilde{\mathcal A}^-
\end{pmatrix}.
\end{equation}

\par Yet another bit of terminology and notation will be needed. Given an integer $d\ge0$ and a vector 
\begin{equation}\label{eqn:albevector}
(\al,\be):=(\al_0,\al_1,\dots,\al_d,\be_1,\dots,\be_d)\in\R^{2d+1},
\end{equation}
we say that $(\al,\be)$ is a {\it plus-vector} if 
\begin{equation}\label{eqn:defplusvector}
\al_0+\sum_{k=1}^d\left(\al_k\cos kt-\be_k\sin kt\right)\ge0
\quad\text{\rm for all }\,t\in(-\pi,\pi].
\end{equation}
Also, for a subspace $V\subset\R^{2d+1}$, we define its {\it plus-dimension} $\text{\rm dim}_+V$ as the maximum number of linearly independent plus-vectors in $V$. 

\par Specifically, relevant to our problem is the subspace 
\begin{equation}\label{eqn:nullspacetil}
\widetilde{\mathcal N}:=\text{\rm ker}\,\widetilde{\mathfrak M}
\end{equation}
(in $\R^{2\widetilde m+1}$), the kernel of the linear map $\widetilde{\mathfrak M}:\R^{2\widetilde m+1}\to\R^{2M}$ given by \eqref{eqn:blockmatrixtil}. The main result of this section can now be stated as follows. 

\begin{thm}\label{thm:exposedpoints} Suppose that $p\in\mathcal P(\La)$ and $\|p\|_1=1$. Then $p$ is an exposed point of $\text{\rm ball}(\mathcal P(\La))$ if and only if $\text{\rm dim}_+\widetilde{\mathcal N}=1$.
\end{thm}

\par To make the analogy between this result and Theorem \ref{thm:extremepoints} more transparent, we may invoke the rank-nullity identity (see, e.g., \cite[p.\,63]{Axl}) to rewrite the condition
\begin{equation}\label{eqn:rankmeqtwom}
\text{\rm rank}\,\mathfrak M=2m 
\end{equation}
from that theorem as $\text{\rm dim}\,\mathcal N=1$. Here, $\mathcal N$ is the kernel of the linear map $\mathfrak M:\R^{2m+1}\to\R^{2M}$ with matrix \eqref{eqn:blockmatrix}, and $\text{\rm dim}\,\mathcal N$ is the (usual) dimension of $\mathcal N$. In the current setting, a suitably adjusted version of \eqref{eqn:rankmeqtwom} provides a sufficient condition for $p$ to be an exposed point. 

\begin{cor}\label{cor:suffcondexposed} Suppose that $p\in\mathcal P(\La)$ and $\|p\|_1=1$. If $\text{\rm rank}\,\widetilde{\mathfrak M}=2\widetilde m$, then $p$ is an exposed point of $\text{\rm ball}(\mathcal P(\La))$.
\end{cor}

\par A few remarks concerning the concepts of plus-vector and plus-dimension are in order. Once again, let $d$ be a nonnegative integer. Along with a given vector \eqref{eqn:albevector} we consider the complex numbers
$$\ga_0=2\al_0,\quad\ga_k=\al_k+i\be_k\quad(k=1,\dots,d)$$
and then extend this collection to a two-sided sequence $\{\ga_k\}_{k\in\Z}$ by setting $\ga_{-k}=\ov\ga_k$ for $1\le k\le d$ and $\ga_k=0$ for $|k|>d$. In this notation, \eqref{eqn:defplusvector} amounts to saying that the (real) trigonometric polynomial
\begin{equation}\label{eqn:tautripol}
\tau(z):=\sum_{k=-d}^d\ga_kz^k\qquad(z\in\T)
\end{equation}
satisfies $\tau(z)\ge0$ everywhere on the circle; indeed, the left-hand side of \eqref{eqn:defplusvector} equals $\f12\tau(e^{it})$. Thus, plus-vectors are essentially the coefficient vectors of nonnegative trigonometric polynomials. In terms of the associated sequence $\{\ga_k\}$, they are characterized (see \cite{Akh}) by the familiar condition that $\{\ga_k\}$ is positive definite, meaning that $\sum_{j,k\ge0}\ga_{j-k}\xi_j\ov\xi_k\ge0$ for any finite sequence $\{\xi_k\}$ of complex numbers. 

\par It is easy to find a subspace $V$ of $\R^{2d+1}$ (for some, or any, $d$) with the property that $\text{\rm dim}_+V\ne\text{\rm dim}\,V$. To give a trivial example with $d=1$, let $V_1$ be the one-dimensional subspace in $\R^3$ spanned by the vector $(0,1,0)$. Then $V_1$ does not contain any nonzero plus-vector, so that $\text{\rm dim}_+V_1=0$, while $\text{\rm dim}\,V_1=1$. 

\par A more interesting example (with $d=2$), which is also more relevant to our topic, can be produced as follows. Let $V_2$ be the two-dimensional subspace in $\R^5$ spanned by the vectors 
$$v_1:=(1,0,-1,0,0)\quad\text{\rm and}\quad v_2:=(0,1,0,0,0).$$
The trigonometric polynomials representing $v_1$ and $v_2$ (in the sense of the above procedure, which leads from \eqref{eqn:albevector} to \eqref{eqn:tautripol}) are 
$$\tau_1(z)=-z^{-2}+2-z^2\quad\text{\rm and}\quad\tau_2(z)=z^{-1}+z,$$
respectively. Since $\tau_1(z)=|z^2-1|^2\ge0$ on $\T$, whereas $\tau_2(z)=2\text{\rm Re}\,z$ changes sign on $\T$, we see that $v_1$ is a plus-vector and $v_2$ is not. Moreover, given real numbers $c_1$ and $c_2$, the linear combination $c_1v_1+c_2v_2$ will be a plus-vector if and only if $c_1\ge0$ and $c_2=0$. Indeed, whenever $c_2\ne0$, the corresponding trigonometric polynomial $c_1\tau_1+c_2\tau_2$ takes values of opposite signs (namely, $2c_2$ and $-2c_2$) at the points $1$ and $-1$. This means that the only plus-vectors in $V_2$ are scalar multiples of $v_1$, and so $\text{\rm dim}_+V_2=1$. 

\par There is one special case where Theorem \ref{thm:exposedpoints} reduces to a simpler criterion (namely, to Theorem \ref{thm:extremepoints} from the preceding section) that does not involve the concept of plus-dimension. 

\begin{prop} Let $p\in\mathcal P(\La)$ and $\|p\|_1=1$. Assume, in addition, that $p$ has no multiple zeros on $\T$. Then $p$ is an exposed point of $\text{\rm ball}(\mathcal P(\La))$ if and only if it is an extreme point thereof. 
\end{prop}

\begin{proof} We only have to verify the \lq\lq if" part. Because the polynomial's zeros $\ze_j$ lying on $\T$ (if any) are all simple, their multiplicities $\la_j$ do not exceed $1$, and the $\mu_j$'s in \eqref{eqn:mujlaj} are all null. Hence $\mu=0$ and $\widetilde m=m$. The polynomial $G_0$, as defined by \eqref{eqn:gnought}, reduces then to the constant function $1$; it follows that $\widetilde G=G$, $\widetilde R=R$ and eventually $\widetilde{\mathfrak M}=\mathfrak M$. Now, assuming that $p$ is an extreme point of $\text{\rm ball}(\mathcal P(\La))$, we use Theorem \ref{thm:extremepoints} to arrive at \eqref{eqn:rankmeqtwom}. Finally, we rewrite this last condition as $\text{\rm rank}\,\widetilde{\mathfrak M}=2\widetilde m$ and invoke Corollary \ref{cor:suffcondexposed} to infer that $p$ is exposed.
\end{proof} 

\par In subtler cases, however, the full strength of Theorem \ref{thm:exposedpoints} may be needed. We conclude this section by looking at a couple of examples to that effect. 

\begin{exmp} Let 
$$p(z)=\f12\left(1-z^2\right)^2.$$
It is easy to check that $\|p\|_1=1$ and $p\in\mathcal P(\La)$ with $\La=\{0,2,4\}$. When written in the form \eqref{eqn:deflam}, this set $\La$ is determined by taking $N=4$, $M=2$, $k_1=1$ and $k_2=3$. The only zeros of $p$ are $\ze_1=1$ and $\ze_2=-1$, their  multiplicities being $\la_1=\la_2=2$; the corresponding parameters in \eqref{eqn:mujlaj} are then $\mu_1=\mu_2=1$, while those in \eqref{eqn:mums} are $\mu=\widetilde m=2$ and $\widetilde s=0$. The polynomials \eqref{eqn:gnought}, \eqref{eqn:newfactone} and \eqref{eqn:newfacttwo} now take the form 
$$\widetilde G(z)=G_0(z)=-\left(1-z^2\right)^2$$
and 
$$\widetilde R(z)=-\f12.$$
It follows that $\widetilde A(0)=-\f12$, whereas the numbers $\widetilde A(k)$ (resp., $\widetilde B(k)$) are null for all $k\in\Z\setminus\{0\}$ (resp., for all $k\in\Z$). Using this to compute the matrix entries \eqref{eqn:abplustil} and \eqref{eqn:abminustil} with the appropriate values of $j$ and $l$, we eventually find that 
$$\widetilde{\mathfrak M}:=-\f12
\begin{pmatrix}
0 & 1 & 0 & 0 &0\\
0 & 1 & 0 & 0 &0\\
0 & 0 & 0 & -1 &0\\
0 & 0 & 0 & 1 &0
\end{pmatrix}.
$$
The rank of this matrix is $2$, so its kernel $\widetilde{\mathcal N}$ has dimension $3$. In fact, we also have $\text{\rm dim}_+\widetilde{\mathcal N}=3$, because the vectors 
$$(1,0,0,0,0),\quad(1,0,1,0,0)\quad\text{\rm and}\quad(1,0,0,0,1)$$
are linearly independent plus-vectors from $\widetilde{\mathcal N}$. An application of Theorem \ref{thm:exposedpoints} now shows that $p$ is not an exposed point of $\text{\rm ball}(\mathcal P(\La))$. At the same time, from Theorem A we see that $p$ is an extreme point of $\text{\rm ball}(\mathcal P_4)$ and hence of $\text{\rm ball}(\mathcal P(\La))$.
\end{exmp}

\begin{exmp} Let 
$$p(z)=c(1-z)^2(2+z),$$
where the number $c>0$ is chosen so as to make $\|p\|_1=1$. Equivalently, we have $p(z)=c\left(2-3z+z^3\right)$, so that $p\in\mathcal P(\La)$ with $\La=\{0,1,3\}$. To write this set $\La$ in the form \eqref{eqn:deflam}, we take $N=3$, $M=1$ and $k_1=2$. The other relevant parameters are $\widetilde m=\mu=1$ and $\widetilde s=1$. Furthermore, the polynomials \eqref{eqn:gnought}, \eqref{eqn:newfactone} and \eqref{eqn:newfacttwo} take the form 
$$\widetilde G(z)=G_0(z)=-(1-z)^2$$
and 
$$\widetilde R(z)=-c(2+z).$$
Using the coefficients of the latter polynomial to compute the entries \eqref{eqn:abplustil} and \eqref{eqn:abminustil}, we find that 
$$\widetilde{\mathfrak M}:=-2c
\begin{pmatrix}
1 & 1 & 0\\
0 & 0 & 1
\end{pmatrix}.
$$
\end{exmp}
The kernel $\widetilde{\mathcal N}$ of this matrix is one-dimensional; moreover, it is spanned by the vector $(1,-1,0)$ which is a plus-vector. Consequently, we have $\text{\rm dim}_+\widetilde{\mathcal N}=1$, and an application of Theorem \ref{thm:exposedpoints} reveals that $p$ is an exposed point of $\text{\rm ball}(\mathcal P(\La))$. (Alternatively, since $\text{\rm rank}\,\widetilde{\mathfrak M}=2$, we may arrive at the same conclusion via Corollary \ref{cor:suffcondexposed}.) One final observation is that $p$ fails to be exposed in $\text{\rm ball}(\mathcal P_3)$, as Theorem B shows. 

\section{Preliminaries}

This section contains three lemmas (at least two of them known) that will be employed later on. Below, we write $L^\infty_\R$ for the set of {\it real-valued} functions in $L^\infty=L^\infty(\T)$, the space of essentially bounded functions on $\T$. 

\begin{lem}\label{lem:charextreme} Let $X$ be a subspace of $L^1$. Suppose also that $f\in X$ is a function with $\|f\|_1=1$ that does not vanish a.e. on $\T$. The following conditions are equivalent. 
\par{\rm (i.1)} $f$ is an extreme point of $\text{\rm ball}(X)$. 
\par{\rm (ii.1)} Whenever $h\in L^\infty_\R$ and $fh\in X$, we have $h=\const$ a.e. on $\T$.
\end{lem}

A proof can be found in \cite [Chapter V, Section 9]{Gam}, where the case $X=H^1$ was treated (but without using any specific properties of $H^1$). The same argument works for an arbitrary subspace $X\subset L^1$ as well.

\begin{lem}\label{lem:charexposed} Under the assumptions of the preceding lemma, the following statements are equivalent. 
\par{\rm (i.2)} $f$ is an exposed point of $\text{\rm ball}(X)$. 
\par{\rm (ii.2)} Whenever $h$ is a nonnegative measurable function on $\T$ for which $fh\in X$, we have $h=\const$ a.e.
\end{lem}

This result appears---in a slightly more general form---as Lemma 1(B) in \cite{DMRL2000}. In fact, for $X=H^1$, the equivalence between (i.2) and (ii.2) is also implied by de Leeuw and Rudin's work in \cite[Subsection 4.2]{dLR}, even though their wording and notation may differ somewhat from ours. Moreover, their reasoning carries over to a generic subspace $X$ of $L^1$. 

\par Our last lemma generalizes the classical fact (see \cite[p.\,92]{G}) that any nonnegative function in the Hardy class $H^{1/2}$ is constant. Here, $H^{1/2}$ can be defined as the closure of $H^1$ in $L^{1/2}=L^{1/2}(\T)$. 

\begin{lem}\label{lem:onehalf} Given an integer $k\ge0$, suppose that $f\in H^{1/2}$ and $fz^{-k}\ge0$ a.e. on $\T$. Then $f\in\mathcal P_{2k}$ (i.e., $f$ is a polynomial of degree at most $2k$).
\end{lem}

\begin{proof} We may assume that $f\not\equiv0$. A standard factorization theorem for Hardy spaces (see \cite[Chapter II]{G}) tells us that $f$ has the form $Bg^2$, where $B$ is a Blaschke product and $g\in H^1$. In particular, since $|B|=1$, we have 
\begin{equation}\label{eqn:modfone}
|f|=|g|^2=g\ov g
\end{equation}
(the identities involved are always assumed to hold a.e. on $\T$). On the other hand, the hypothesis that $fz^{-k}\ge0$ yields 
\begin{equation}\label{eqn:modftwo}
|f|=fz^{-k}=Bg^2z^{-k}.
\end{equation}
Comparing \eqref{eqn:modfone} and \eqref{eqn:modftwo}, we see that $g\ov g=Bg^2z^{-k}$, or equivalently, 
\begin{equation}\label{eqn:antianal}
\ov g=Bgz^{-k}.
\end{equation}
Now, the functions $g$ and $Bg$ are both in $H^1$, so their spectra are contained in $[0,\infty)$. It follows that 
\begin{equation}\label{eqn:twospectra}
\text{\rm spec}\,\ov g\subset(-\infty,0]\quad\text{\rm and}\quad
\text{\rm spec}\left(Bgz^{-k}\right)\subset[-k,\infty).
\end{equation}
At the same time, the two spectra in \eqref{eqn:twospectra} are equal by virtue of \eqref{eqn:antianal}, so they are actually contained in $[-k,0]$. This in turn implies that 
$$\text{\rm spec}\,g\subset[0,k]\quad\text{\rm and}\quad
\text{\rm spec}(Bg)\subset[0,k].$$
In other words, $g$ and $Bg$ are both in $\mathcal P_k$. Consequently, their product (which is $f$) lies in $\mathcal P_{2k}$, as required.
\end{proof}

\section{Proof of Theorem \ref{thm:extremepoints}}

Let $p\in\mathcal P(\La)$ and $\|p\|_1=1$. In view of Lemma \ref{lem:charextreme}, the issue boils down to deciding whether $p$ can be multiplied by a nonconstant function $h\in L^\infty_\R$ to produce another polynomial in $\mathcal P(\La)$. 
\par Our method consists essentially in parametrizing the class of eligible functions $h$. So let us assume that $h\in L^\infty_\R$ and the product $ph=:q$ is in $\mathcal P(\La)$. We have then 
\begin{equation}\label{eqn:hqpqp}
h=\f qp=\f{\ov q}{\ov p}=\f{z^N\ov q}{z^N\ov p}=\f{q^*}{p^*}
\end{equation}
on $\T$, whence in particular 
\begin{equation}\label{eqn:ptimesqstar}
pq^*=p^*q.
\end{equation}
This last identity must actually hold everywhere in $\C$, since both sides are polynomials. 

\par Next, we recall the factorization $p=GR$, where the two factors are defined by  \eqref{eqn:ggg} and \eqref{eqn:rrr}, and we go on to claim that the ratio 
\begin{equation}\label{eqn:qqr}
Q:=\f qR
\end{equation}
is a polynomial. To see why, let $\ze_1,\dots,\ze_\nu$ be the distinct zeros of $p$ lying on $\T$, and let $\la_1,\dots,\la_\nu$ be their respective multiplicities (we stick to the notation of Section 3). The $\ze_j$'s are then also zeros for $p^*$, with the same multiplicities. Consequently, the polynomial 
\begin{equation}\label{eqn:polyzercirc}
\mathcal T(z):=\prod_{j=1}^\nu(z-\ze_j)^{\la_j}
\end{equation}
divides both $p$ and $p^*$, while the product 
\begin{equation}\label{eqn:defcapphi}
G\mathcal T=:\Phi
\end{equation}
is the {\it greatest common divisor} (GCD) of $p$ and $p^*$. For future reference, we write down the latter fact as 
\begin{equation}\label{eqn:phigcd}
\Phi=\text{\rm GCD}\,(p,p^*).
\end{equation}
We further remark that $\mathcal T$ divides $R$ (because $R=p/G$ and $G$ has no zeros on $\T$), so that 
\begin{equation}\label{eqn:defcappsi}
R/\mathcal T=:\Psi
\end{equation}
is a polynomial and 
$$p=G\mathcal T\cdot\f R{\mathcal T}=\Phi\Psi.$$ 
Plugging this into \eqref{eqn:ptimesqstar} yields
\begin{equation}\label{eqn:psiq}
\Psi q^*=\f{p^*}{\Phi}q,
\end{equation}
and since the polynomials $p/\Phi(=\Psi)$ and $p^*/\Phi$ are relatively prime by virtue of \eqref{eqn:phigcd}, it follows from \eqref{eqn:psiq} that $\Psi$ divides $q$. At the same time, $q$ is divisible by $\mathcal T$, since otherwise the ratio $q/p(=h)$ would not be essentially bounded on $\T$. Finally, because the polynomials $\Psi$ and $\mathcal T$ are relatively prime (indeed, they have disjoint zero sets) and each of them divides $q$, we conclude that $q$ is divisible by their product, which is $R$. This proves our claim that the function $Q$ in \eqref{eqn:qqr} is a polynomial. 

\par Going back to the identity $h=q/p$, we now combine it with the equalities 
$p=GR$ and $q=QR$ to find that 
\begin{equation}\label{eqn:hqg}
h=Q/G.
\end{equation}
Furthermore, we have the elementary formula 
\begin{equation}\label{eqn:elemformgz}
G(z)=z^m\prod_{j=1}^n|z-a_j|^{2m_j},\qquad z\in\T
\end{equation}
(which holds because $1-\ov a_jz=z(\ov z-\ov a_j)$ for all $j$ and all $z\in\T$), and together with \eqref{eqn:hqg} this shows that 
\begin{equation}\label{eqn:zmqhprod}
z^{-m}Q(z)=h(z)\prod_{j=1}^n|z-a_j|^{2m_j},\qquad z\in\T.
\end{equation}
Consequently, the function $z\mapsto z^{-m}Q(z)$ is {\it real-valued} on $\T$, so its spectrum, $\text{\rm spec}(z^{-m}Q)$, is symmetric with respect to the origin. This function is also a {\it trigonometric polynomial} (because $Q$ is an analytic polynomial, as explained above); and since $\text{\rm spec}\,Q\subset[0,\infty)$, it follows easily that 
$$\text{\rm spec}(z^{-m}Q)\subset[-m,m].$$
The coefficients
\begin{equation}\label{eqn:defofdk}
d_l:=\widehat{(z^{-m}Q)}(l)=\widehat Q(l+m),\qquad l\in\Z,
\end{equation}
are therefore null for $|l|>m$ and satisfy the relations 
\begin{equation}\label{eqn:symcoef}
d_{-l}=\ov d_l\quad\text{\rm for}\quad l=0,\dots,m 
\end{equation}
(whence, in particular, $d_0\in\R$). We have then 
\begin{equation}\label{eqn:qofzeqsum}
Q(z)=\sum_{l=0}^{2m}d_{l-m}z^l,
\end{equation}
and there are further restrictions on the $d_l$'s coming from the fact that the polynomial $q=QR$ is in $\mathcal P(\La)$. 

\par To make these explicit, let us compute the Fourier coefficients $\widehat q(k)$ in terms of the numbers \eqref{eqn:defofck} and \eqref{eqn:defofdk}. For a fixed $k\in\Z$, we get
\begin{equation}\label{eqn:compqhatk}
\begin{aligned}
\widehat q(k)&=\sum_{l=0}^{2m}\widehat R(k-l)\widehat Q(l)
=\sum_{l=0}^{2m}C_{k-l}\,d_{l-m}
=\sum_{l=-m}^{m}C_{k-l-m}\,d_l\\
&=\sum_{l=0}^{m}C_{k-l-m}\,d_l+\sum_{l=1}^{m}C_{k+l-m}\,\ov d_l,
\end{aligned}
\end{equation}
where the last step relies on \eqref{eqn:symcoef}. Now, since $q\in\mathcal P(\La)$, we have the conditions 
\begin{equation}\label{eqn:qhatkjzero}
\widehat q(k_j)=0\quad\text{\rm for}\quad j=1,\dots,M,
\end{equation}
and \eqref{eqn:compqhatk} allows us to rewrite them as 
\begin{equation}\label{eqn:qhatkjnull}
\sum_{l=0}^{m}C_{k_j-l-m}\,d_l+\sum_{l=1}^{m}C_{k_j+l-m}\,\ov d_l=0.
\end{equation}
We now introduce the {\it real} parameters $\al_0,\al_1,\dots,\al_m$ and $\be_1,\dots,\be_m$ setting 
\begin{equation}\label{eqn:dalbe}
d_0=2\al_0,\quad d_l=\al_l+i\be_l\quad\text{\rm for}\quad l=1,\dots,m.
\end{equation}
(This done, \eqref{eqn:qofzeqsum} takes the form 
\begin{equation}\label{eqn:qofzeqre}
Q(z)=2z^m\left\{\al_0+\text{\rm Re}\,\sum_{l=1}^{m}(\al_l+i\be_l)z^l\right\},
\qquad z\in\T,
\end{equation}
a formula to be noted for later reference.) Also, in accordance with \eqref{eqn:defakbk} we write 
\begin{equation}\label{eqn:cab}
C_r=A(r)+iB(r),\qquad r\in\Z,
\end{equation}
with $A(r)$ and $B(r)$ real. Plugging \eqref{eqn:dalbe} and \eqref{eqn:cab} into \eqref{eqn:qhatkjnull}, we then split each of the resulting equations into a real and imaginary part to obtain 
\begin{equation}\label{eqn:reparteq}
\sum_{l=0}^mA^+_{j,l}\,\al_l+\sum_{l=1}^mB^-_{j,l}\,\be_l=0\qquad(j=1,\dots,M)
\end{equation}
and 
\begin{equation}\label{eqn:imparteq}
\sum_{l=0}^mB^+_{j,l}\,\al_l-\sum_{l=1}^mA^-_{j,l}\,\be_l=0\qquad(j=1,\dots,M),
\end{equation}
where the notations \eqref{eqn:abplus} and \eqref{eqn:abminus} have been used. 

\par The system of $2M$ real equations \eqref{eqn:reparteq}\,\&\,\eqref{eqn:imparteq}, which has thus emerged, means that the vector 
\begin{equation}\label{eqn:vectoralbe}
(\al,\be):=(\al_0,\al_1,\dots,\al_m,\be_1,\dots,\be_m)
\end{equation}
(or rather the column vector $(\al,\be)^{\rm T}$, where the superscript $\text{\rm T}$ denotes transposition) belongs to the subspace 
\begin{equation}\label{eqn:nullspace}
\mathcal N:=\text{\rm ker}\,\mathfrak M,
\end{equation}
the kernel of the linear map $\mathfrak M:\R^{2m+1}\to\R^{2M}$ defined by \eqref{eqn:blockmatrix}. 

\par To summarize, every function $h\in L^\infty_\R$ for which $ph\in\mathcal P(\La)$ is given by \eqref{eqn:hqg}, where $Q$ is a polynomial of the form \eqref{eqn:qofzeqre} whose coefficient vector \eqref{eqn:vectoralbe} satisfies $(\al,\be)\in\mathcal N$. Conversely, for every vector \eqref{eqn:vectoralbe} from $\mathcal N$, we may consider the associated polynomial \eqref{eqn:qofzeqre} and use it to define the function $h$ by \eqref{eqn:hqg}. This $h$ will be in $L^\infty_\R$, thanks to \eqref{eqn:zmqhprod}, and the polynomial $q=QR(=ph)$ will be in $\mathcal P(\La)$. Indeed, from \eqref{eqn:degrles} and \eqref{eqn:qofzeqsum} it follows that $q\in\mathcal P_N$, while the conditions \eqref{eqn:qhatkjzero} are ensured by \eqref{eqn:reparteq} and \eqref{eqn:imparteq}. 

\par The constant function $h=1$ corresponds to the choice $Q=G$; the associated coefficient vector, say $(\al_G,\be_G)$, is then a nonzero element of $\mathcal N$, so we always have $\text{\rm dim}\,\mathcal N\ge1$. Moreover, it is now clear that a {\it nonconstant} function $h\in L^\infty_\R$ with $ph\in\mathcal P(\La)$ can be found if and only if $\text{\rm dim}\,\mathcal N>1$. (In fact, for such an $h$ to exist, there should be a vector $(\al,\be)\in\mathcal N$ other than a scalar multiple of $(\al_G,\be_G)$.) Consequently, in view of Lemma \ref{lem:charextreme}, this last condition characterizes the non-extreme points $p$. In other words, the extreme points of $\text{\rm ball}(\mathcal P(\La))$ are precisely those unit-norm polynomials $p\in\mathcal P(\La)$ for which $\text{\rm dim}\,\mathcal N=1$. 

\par Finally, because the rank-nullity theorem tells us that 
$$\text{\rm rank}\,\mathfrak M+\text{\rm dim}\,\mathcal N=2m+1,$$
the criterion just found can also be stated in the form $\text{\rm rank}\,\mathfrak M=2m$. The proof is complete.

\section{Proofs of Theorem \ref{thm:exposedpoints} and Corollary \ref{cor:suffcondexposed}}

\noindent{\it Proof of Theorem \ref{thm:exposedpoints}.} Let $p\in\mathcal P(\La)$ and $\|p\|_1=1$. This time, in view of Lemma \ref{lem:charexposed}, we need to determine whether $p$ can be multiplied by a nonconstant function $h\ge0$ to produce another polynomial in $\mathcal P(\La)$. So let us assume that $h$ is a nonnegative function on $\T$ and the product $ph=:q$ is in $\mathcal P(\La)$. This brings us to identities \eqref{eqn:hqpqp} exactly as before (indeed, to verify them, one only uses the fact that $h$ is real-valued). As a consequence, \eqref{eqn:ptimesqstar} holds true on $\T$ and hence everywhere in $\C$. 

\par In what follows, we retain the notation established in Sections 2 and 3 above. In addition to the various polynomials (such as $G$, $R$, $G_0$, $\widetilde G$,  $\widetilde R$) that were introduced there, we shall make use of the polynomials $\mathcal T$, $\Phi$ and $\Psi$ coming from the preceding proof; see formulas \eqref{eqn:polyzercirc}, \eqref{eqn:defcapphi} and \eqref{eqn:defcappsi} for definitions. In particular, the identities 
\begin{equation}\label{eqn:takacc}
p=\Phi\Psi=\widetilde G\widetilde R
\end{equation}
should be kept in mind. 

\par We now recall the formula \eqref{eqn:gnoughtbis}, along with the inequalities 
$$0\le2\mu_j\le\la_j,\qquad j=1,\dots,\nu,$$
to see that $G_0$ divides $\mathcal T$. We then look at the polynomial
$$S:=\mathcal T/G_0$$
and note that its zeros (if any) are all simple and contained among the $\ze_j$'s; in fact, the zeros of $S$ are precisely those $\ze_j$'s for which $\la_j$ is odd. The identity 
\begin{equation}\label{eqn:sqsqfree}
\mathcal T=G_0S
\end{equation}
thus provides a (standard) representation of $\mathcal T$ as the product of a square and a square-free factor. Furthermore, multiplying both sides of \eqref{eqn:sqsqfree} by $G$ yields 
$$\Phi=G\mathcal T=GG_0S=\widetilde GS,$$
and we may combine the resulting equality $\Phi=\widetilde GS$ with \eqref{eqn:takacc} to infer that 
\begin{equation}\label{eqn:spsieqrtil}
\widetilde R=S\Psi.
\end{equation}

\par On the other hand, going back to \eqref{eqn:ptimesqstar} and rewriting it as \eqref{eqn:psiq}, we couple this with \eqref{eqn:phigcd} to deduce (just as we did in the preceding section) that $\Psi$ divides $q$. The ratio $q/\Psi=:Q_0$ is therefore a polynomial. Yet another function we need is 
\begin{equation}\label{eqn:defqtil}
\widetilde Q:=\f{Q_0}S\left(=\f q{S\Psi}\right).
\end{equation}
Using \eqref{eqn:takacc}, \eqref{eqn:spsieqrtil} and \eqref{eqn:defqtil}, we now obtain
$$h=\f qp=\f q{\widetilde G\widetilde R}=\f q{\widetilde GS\Psi}
=\f{\widetilde Q}{\widetilde G},$$
so that
\begin{equation}\label{eqn:hqtilovergtil}
h=\widetilde Q/\widetilde G.
\end{equation}
Also, from \eqref{eqn:spsieqrtil} and \eqref{eqn:defqtil} we see that 
\begin{equation}\label{eqn:qqrtiltil}
q=\widetilde Q\widetilde R.
\end{equation}

\par This said, let us pause to take a closer look at $\widetilde Q$. Because $Q_0$ and $S$ are both polynomials, whereas the zeros of $S$ are all simple and contained in $\T$, it follows that $\widetilde Q(=Q_0/S)$ belongs to every Hardy space $H^{1-\varepsilon}$ with $0<\varepsilon<1$. In particular, 
\begin{equation}\label{eqn:honehalf}
\widetilde Q\in H^{1/2}. 
\end{equation}
Besides, we claim that 
\begin{equation}\label{eqn:qtilpos}
z^{-\widetilde m}\widetilde Q(z)\ge0,\qquad z\in\T,
\end{equation}
a fact to be verified in a moment. Indeed, we know from \eqref{eqn:elemformgz} that 
$z^{-m}G(z)\ge0$ for all $z\in\T$; similarly, the identity 
$$G_0(z)=z^\mu\prod_{j=1}^\nu|z-\ze_j|^{2\mu_j},\qquad z\in\T,$$
implies that $z^{-\mu}G_0(z)\ge0$ on $\T$. The two inequalities together yield 
\begin{equation}\label{eqn:uhogorlonos}
z^{-\widetilde m}\widetilde G(z)=z^{-m}G(z)\cdot z^{-\mu}G_0(z)\ge0
\end{equation}
(recall that $\widetilde m=m+\mu$ and $\widetilde G=GG_0$). At the same time, \eqref{eqn:hqtilovergtil} tells us that 
$\widetilde Q=\widetilde Gh$ and so
\begin{equation}\label{eqn:utkonos}
z^{-\widetilde m}\widetilde Q(z)=z^{-\widetilde m}\widetilde G(z)h(z).
\end{equation}
Finally, our claim \eqref{eqn:qtilpos} follows from \eqref{eqn:utkonos}, thanks to \eqref{eqn:uhogorlonos} and the hypothesis that $h\ge0$. 

\par Now that we have \eqref{eqn:honehalf} and \eqref{eqn:qtilpos} at our disposal, an application of Lemma \ref{lem:onehalf} shows that $\widetilde Q\in\mathcal P_{2\widetilde m}$. Consequently, the function $z\mapsto z^{-\widetilde m}\widetilde Q(z)$ is a {\it nonnegative} trigonometric polynomial with spectrum in $[-\widetilde m,\widetilde m]$. Denoting its $l$th Fourier coefficient by $\widetilde d_l$, we have 
\begin{equation}\label{eqn:qtileqsum}
\widetilde Q(z)=\sum_{l=0}^{2\widetilde m}\widetilde d_{l-\widetilde m}z^l.
\end{equation}
We also know that $\widetilde d_{-l}=\ov{\widetilde d_l}$ for all $l\in\Z$. Setting 
\begin{equation}\label{eqn:dtilalbe}
\widetilde d_0=2\widetilde \al_0\quad\text{\rm and}\quad
\widetilde d_l=\widetilde\al_l+i\widetilde\be_l\quad(l=1,\dots,\widetilde m),
\end{equation}
with $\widetilde\al_l$ and $\widetilde\be_l$ real, we now rewrite \eqref{eqn:qtileqsum} in terms of these real parameters to get
\begin{equation}\label{eqn:qtileqre}
\widetilde Q(z)=2z^{\widetilde m}\left\{\widetilde\al_0
+\text{\rm Re}\,\sum_{l=1}^{\widetilde m}(\widetilde\al_l+i\widetilde\be_l)z^l\right\},
\qquad z\in\T.
\end{equation}
We may therefore rephrase condition \eqref{eqn:qtilpos} by saying that the vector 
\begin{equation}\label{eqn:albevectortil}
(\widetilde\al,\widetilde\be):=(\widetilde\al_0,\widetilde\al_1,\dots,
\widetilde\al_{\widetilde m},\widetilde\be_1,\dots,\widetilde\be_{\widetilde m})
\end{equation}
is a plus-vector in $\R^{2\widetilde m+1}$. 

\par To gain full information about the $\widetilde\al_l$'s and $\widetilde\be_l$'s, we should also recall the relation \eqref{eqn:qqrtiltil} and use the fact that $q\in\mathcal P(\La)$. This amounts to recasting conditions \eqref{eqn:qhatkjzero} in terms of the coefficients of $\widetilde Q$ and $\widetilde R$. Eventually, these conditions take the form 
\begin{equation}\label{eqn:reparteqtil}
\sum_{l=0}^{\widetilde m}\widetilde A^+_{j,l}\,\widetilde\al_l+\sum_{l=1}^{\widetilde m}\widetilde B^-_{j,l}\,\widetilde\be_l=0\qquad(j=1,\dots,M)
\end{equation}
and 
\begin{equation}\label{eqn:imparteqtil}
\sum_{l=0}^{\widetilde m}\widetilde B^+_{j,l}\,\widetilde\al_l-\sum_{l=1}^{\widetilde m}\widetilde A^-_{j,l}\,\widetilde\be_l=0\qquad(j=1,\dots,M),
\end{equation}
where the notations \eqref{eqn:abplustil} and \eqref{eqn:abminustil} are being used. (The calculations leading to \eqref{eqn:reparteqtil} and \eqref{eqn:imparteqtil} are almost identical to those made in the preceding section, when moving from \eqref{eqn:qhatkjzero} to \eqref{eqn:reparteq} and \eqref{eqn:imparteq}. We only have to replace the factorization $q=QR$, used at that stage, by $q=\widetilde Q\widetilde R$ and proceed accordingly with the coefficients of the polynomials involved. Formally, we just switch to the tilde notation wherever applicable.) 

\par Now, equations \eqref{eqn:reparteqtil} and \eqref{eqn:imparteqtil} tell us that the vector \eqref{eqn:albevectortil}, or rather the column vector $(\widetilde\al,\widetilde\be)^{\rm T}$, belongs to the subspace $\widetilde{\mathcal N}$, defined (in accordance with \eqref{eqn:nullspacetil}) as the kernel of the linear map $\widetilde{\mathfrak M}:\R^{2\widetilde m+1}\to\R^{2M}$ with matrix \eqref{eqn:blockmatrixtil}. 

\par In summary, every nonnegative function $h$ with $ph\in\mathcal P(\La)$ is given by \eqref{eqn:hqtilovergtil}, where $\widetilde Q$ is a polynomial of the form \eqref{eqn:qtileqre} such that \eqref{eqn:albevectortil} is a plus-vector from $\widetilde{\mathcal N}$. Conversely, for every plus-vector \eqref{eqn:albevectortil} from $\widetilde{\mathcal N}$, we may consider the associated polynomial \eqref{eqn:qtileqre} and then define the function $h$ by \eqref{eqn:hqtilovergtil}. This $h$ will be nonnegative, thanks to \eqref{eqn:qtilpos} and \eqref{eqn:uhogorlonos}, and the polynomial $q=\widetilde Q\widetilde R(=ph)$ will be in $\mathcal P(\La)$. Indeed, from \eqref{eqn:degrtillestil} and \eqref{eqn:qtileqsum} it follows that $q\in\mathcal P_N$, while the conditions \eqref{eqn:qhatkjzero} are ensured by \eqref{eqn:reparteqtil} and \eqref{eqn:imparteqtil}. 

\par The constant function $h=1$ corresponds to the choice $\widetilde Q=\widetilde G$. Letting $(\widetilde\al_*,\widetilde\be_*)$ stand for the coefficient vector \eqref{eqn:albevectortil} that arises in this special case, we see that $(\widetilde\al_*,\widetilde\be_*)$ is a nonzero plus-vector in $\widetilde{\mathcal N}$. Thus, we always have 
\begin{equation}\label{eqn:alwaystil}
\text{\rm dim}_+\widetilde{\mathcal N}\ge1.
\end{equation}
Moreover, the above discussion reveals that a {\it nonconstant} function $h\ge0$ with $ph\in\mathcal P(\La)$ can be found if and only if $\widetilde{\mathcal N}$ contains plus-vectors that are not scalar multiples of $(\widetilde\al_*,\widetilde\be_*)$. This last condition, meaning that $\text{\rm dim}_+\widetilde{\mathcal N}>1$, is therefore equivalent to saying that $p$ is a non-exposed point of $\text{\rm ball}(\mathcal P(\La))$. In other words, the exposed points are characterized by the property that $\text{\rm dim}_+\widetilde{\mathcal N}=1$. The proof is complete. \qed

\medskip\noindent{\it Proof of Corollary \ref{cor:suffcondexposed}.} We know from the rank-nullity theorem that 
$$\text{\rm rank}\,\widetilde{\mathfrak M}+\text{\rm dim}\,\widetilde{\mathcal N}
=2\widetilde m+1.$$
Consequently, the hypothesis $\text{\rm rank}\,\widetilde{\mathfrak M}=2\widetilde m$ implies that $\text{\rm dim}\,\widetilde{\mathcal N}=1$ and hence $\text{\rm dim}_+\widetilde{\mathcal N}\le1$. In conjunction with \eqref{eqn:alwaystil}, this means that $\text{\rm dim}_+\widetilde{\mathcal N}$ is actually equal to $1$, and the result follows by Theorem \ref{thm:exposedpoints}. \qed

\end{document}